\documentclass[3p,preprint]{elsarticle}

\makeatletter
\def\ps@pprintTitle{%
 \let\@oddhead\@empty
 \let\@evenhead\@empty
 \def\@oddfoot{\centerline{\thepage}}%
 \let\@evenfoot\@oddfoot}
\makeatother

\usepackage{graphicx}
\usepackage{amsmath}
\usepackage{amsmath}
\usepackage{amssymb}
\usepackage{amsthm}
\usepackage{color}
\usepackage{subfigure}
\usepackage{url}
\usepackage[table]{xcolor}
\usepackage{arydshln}

\newcommand{\cL}{\mathcal{L}}
\newcommand{\R}{\mathbb{R}}
\DeclareMathOperator{\trace}{trace}

\newcommand{\seclabel}[1]{\label{sec:#1}}
\newcommand{\secref}[1]{Section~\ref{sec:#1}}
\newcommand{\applabel}[1]{\label{app:#1}}
\newcommand{\appref}[1]{\ref{app:#1}}
\newcommand{\eqlabel}[1]{\label{eq:#1}}
\newcommand{\Eqref}[1]{Equation~\eqref{eq:#1}}
\newcommand{\figlabel}[1]{\label{fig:#1}}
\newcommand{\figref}[1]{Figure~\ref{fig:#1}}
\newcommand{\rref}[1]{Reference~\cite{#1}}
\newcommand{\rrefs}[1]{References~\cite{#1}}

\newcommand{\thmlabel}[1]{\label{thm:#1}}
\newcommand{\thmref}[1]{Theorem~\ref{thm:#1}}
\newcommand{\qlabel}[1]{\label{thm:#1}}
\newcommand{\qref}[1]{Question~\ref{thm:#1}}

\newcommand{\tablabel}[1]{\label{tab:#1}}
\newcommand{\tabref}[1]{Table~\ref{tab:#1}}

\newcommand\scalemath[2]{\scalebox{#1}{\mbox{\ensuremath{\displaystyle #2}}}}
\newcommand\s{\cellcolor{black!10}}

\newtheorem{theorem}{Theorem}
\newtheorem{lemma}{Lemma}
\newtheorem{corollary}{Corollary}

\theoremstyle{remark}
\newtheorem{remark}{Remark}[section]

\theoremstyle{definition}

\newtheorem{question}{Question}

\begin{document}

\begin{frontmatter}

  \title{An algorithmic exploration of the existence of high-order
    summation by parts operators with diagonal norm}

\author[ksu]{Nathan Albin\corref{cor}}
\ead{albin@math.ksu.edu}

\author[ksu]{Joshua Klarmann}
\ead{jklarm@gmail.com}

\address[ksu]{Department of Mathematics, Kansas State University, 138
  Cardwell Hall, Manhattan, KS 66506}

\cortext[cor]{Corresponding author}

\begin{abstract}
  This paper explores a common class of diagonal-norm summation by
  parts (SBP) operators found in the literature, which can be
  parameterized by an integer triple $(s,t,r)$ representing the
  interior order of accuracy ($2s)$, the boundary order of accuracy
  ($t$), and the dimension of the boundary closure ($r$).  There is no
  simple formula for determining whether or not an SBP operator exists
  for a given triple of parameters.  Instead, one must check that
  certain compatibility conditions are met: namely that a particular
  linear system of equations has a positive solution.  Partly because
  of the complexity involved, not much is known about diagonal-norm
  SBP operators with $2s>10$.

  By utilizing a new algorithm for answering the question ``Does an
  SBP operator exist for the parameters $(s,t,r)$?'', it is possible
  to explore the existence of SBP operators with high order accuracy,
  and previously unknown SBP operators with interior order of accuracy
  as large as $2s=30$ are found.  Additionally, a method for
  optimizing the spectral radius of the SBP derivative is introduced,
  and the effectiveness of this method is explored through numerical
  experiment.
\end{abstract}

  
\end{frontmatter}

\section{Introduction}
The need for high-order numerical methods in the simulation of
long-distance advection and wave propagation is well established.  The
introduction to a 1972 paper by Kreiss and Oliger~\cite{Kreiss1972},
for example, provides a review of works dating back to the mid-1960's
indicating that the inadequacy of second-order methods for
applications in meteorology and oceanography was becoming apparent
even then.  The argument is based on the observation that the error in
a numerical solution of a hyperbolic partial differential equation
(PDE) scales roughly as a constant times $T\omega^{(p+1)}h^{p}$, where
$p$ is the order of accuracy of the numerical derivative operator, $h$
is the spatial step size, $\omega$ is a characteristic frequency of
the solution, and $T$ is the final time to which the equation is to be
solved.  Thus, in order to maintain a given level of error, the
spatial step size should be scaled as $h \sim
\omega^{-(p+1)/p}\;T^{-1/p}$.  For problems wherein $\omega$ and $T$
are very large, it is essential that $p$ be large as well, to avoid
the need for overly small $h$.

For spatial derivative approximations based on finite differences, the
combination of wide stencils to allow high-order accuracy and biased
stencils to accommodate domain boundaries presents a particular
challenge; it is very unlikely that an arbitrarily chosen high-order
finite difference operator will lead to a stable solver for a
hyperbolic initial-boundary value problem.  Because of this, much
research has been conducted in the search for stable and accurate
finite difference schemes.

\subsection{Summation by parts operators}

Among the various types of high-order finite difference operators, the
summation by parts (SBP) operators are unique in that their
construction incorporates the construction of a natural discrete
energy norm that can be used to prove stability for PDE solvers.
Numerical schemes based on SBP operators have proven effective in
simulating a wide variety of physical phenomena, including fluid
flow~\cite{Svard2007,Svard2008,Osusky2010}, elastic wave
propagation~\cite{Appelo2009,Nilsson2007,Sjogreen2012}, and orbiting
binary black holes~\cite{Pazos2009}.

The basic idea behind SBP operators (see, e.g.,
\rrefs{Kreiss1974,Strand1994,Carpenter1994}) is straightforward.  One
seeks to build, simultaneously, a finite difference operator and an
associated vector norm that mimic, in a semi-discrete setting, some
continuum energy estimate for the PDE.  The one-dimensional advection
equation on a bounded interval provides the canonical example.

Consider the PDE
\begin{equation}
  \eqlabel{pde}
  u_t + u_x = 0 \qquad x\in(0,1),\quad t>0,
\end{equation}
with suitable initial and boundary conditions.  The energy
\begin{equation*}
  \mathcal{E}_u(t) = \|u(t,\cdot)\|_{L^2}^2 = \int_0^1 u(t,x)^2 \;dx,
\end{equation*}
has the property that, for $u$ solving \Eqref{pde}, $\mathcal{E}_u$
satisfies
\begin{equation}
  \eqlabel{con-energy-eq}
  \frac{d\mathcal{E}_u}{dt} = 2\int_0^1u\,u_t\;dx
  = -2\int_0^1 u\,u_x\;dx = -\int_0^1\frac{\partial}{\partial x} u^2\;dx
  = u(t,0)^2 - u(t,1)^2.
\end{equation}

Now, consider the following semi-discrete form of \Eqref{pde}.  In
what follows, for the sake of simplifying formulas, we deviate from
convention and index arrays beginning at $0$.  Let
$\{x_i\}_{i=0}^{n-1}$ be the grid of $n$ equispaced nodes in $[0,1]$
with step size $h=1/(n-1)$, and let $v(t)$ be the $n$-vector
approximating $u$ in the method-of-lines interpretation.  That is,
$v_i(t) \approx u(t,x_i)$ solves the semi-discrete equation
\begin{equation}
  \eqlabel{semi-discrete}
  v_t + D_hv = 0,
\end{equation}
for some $n\times n$ finite difference operator $D_h$.  Emulating the
continuum case, let $P_h$ be an $n\times n$ symmetric positive
definite matrix, and define the energy
\begin{equation*}
  E_v(t) = \|v(t)\|_{P_h}^2 = v(t)^TP_h\,v(t).
\end{equation*}

If $P_h$ and $D_h$ together satisfy the condition
\begin{equation}
  \eqlabel{energy}
  P_hD_h + D_h^TP_h = e_{n-1}e_{n-1}^T - e_0e_0^T = Q,
\end{equation}
with $\{e_i\}_{i=0}^{n-1}$ the canonical basis in $\R^n$, then it is
straightforward to check that if $v$ is a solution to
\Eqref{semi-discrete}, then the energy satisfies
\begin{equation*}
  \frac{dE_v}{dt} = 
  v_t^TP_h\,v + v^TP_h\,v_t
  = -v^T\left(P_hD_h + D_h^TP_h\right)v
  = v_0^2 - v_{n-1}^2,
\end{equation*}
which is a semi-discrete analog of \Eqref{con-energy-eq}.  This
property can be used to prove the stability of the fully discrete
numerical solver.

Thus, the construction of an SBP first derivative operator (actually,
an operator/norm pair) consists of constructing $n\times n$ matrices
$D_h$ and $P_h$ with the following properties.

\begin{center}
  \begin{minipage}{0.8\linewidth}
    \begin{description}
    \item[(P1)] $D_h$ is a finite difference approximation of the
      first derivative.
    \item[(P2)] $P_h$ and $D_h$ together satisfy the energy
      condition~\eqref{eq:energy}.
    \item[(P3)] $P_h$ is a positive definite matrix.
    \end{description}
  \end{minipage}
\end{center}

\subsection{The parameters $(s,t,r)$}

Making use of the scale-invariance of the derivative and norm, namely
that $D_h=\frac{1}{h}D_1$ and $P_h=hP_1$, we now drop the subscript
$h$ and assume a step size of $h=1$.  In their general form,
properties (P1)--(P3) form a large nonlinear system of equations and
inequalities in the $n^2$ entries in $D$ and the $n(n+1)/2$ entries in
$P$.  This is undesirable for three reasons: the equations are
nonlinear, there are many of them, and their solution naturally
depends on the grid size $n$.

Fortunately, each of these problems can be treated by standard
methods~\cite{Kreiss1974,Svard2014}.  Let the positive integer
triple $(s,t,r)$ be given.  For the remainder of the paper, $P$ is
assumed to have the block diagonal form

\begin{equation}
  \eqlabel{P-form}
  P =
  \begin{bmatrix}
    \tilde{P} & 0 & 0\\
    0 & I & 0\\
    0 & 0 & \hat{P}
  \end{bmatrix} = P^T,
\end{equation}
with $\tilde{P}$ and $\hat{P}$ symmetric positive definite $r\times r$
matrices.  Furthermore, $D$ is assumed to be a centered finite
difference operator of order $2s$ in its interior $n-2r$ rows, and a
finite difference operator of order $t$ in its first $r$ and last $r$
rows.  Under these assumptions, $D$ can be factored as
\begin{equation}
  \eqlabel{D-form}
  D =
  \begin{bmatrix}
    \tilde{P}^{-1} & 0 & 0 \\ 
    0 & I & 0 \\
    0 & 0 & \hat{P}^{-1}
  \end{bmatrix}
  \begin{bmatrix}
    B & C_0 & 0 \\
    -C_0^T & \tilde{D} & -\hat{C}_0^T \\
    0 & \hat{C}_0 & \hat{B}
  \end{bmatrix},
\end{equation}
where $B$ and $\hat{B}$ are $r\times r$ block matrices, $C_0$ and
$\hat{C}_0$ are $r\times(n-2r)$ blocks of the forms
\begin{equation*}
  C_0 = \Big[ C \; 0\Big]\quad\text{and}\quad 
  \hat{C}_0 = \Big[ 0 \; \hat{C} \Big],
\end{equation*}
respectively, where $C,\hat{C}\in\R^{r\times s}$.  For example, when
$s=2$ and $r=4$, the matrix on the right-hand side of \Eqref{D-form}
has a banded structure of the form
\begin{equation*}
  \scalemath{0.6}{
    \left[
    \begin{array}{cccc:cccccc}
      \times & \times & \times & \times & \s & \s \\
      \times & \times & \times & \times & \s & \s \\
      \times & \times & \times & \times & \s \times & \s\\
      \times & \times & \times & \times & \s \times & \s \times \\
      \hdashline
      & & \times & \times & 0 & \times & \times\\
      & & & \times & \times & 0 & \times & \times\\
      & & & & \times & \times & 0 & \times & \times\\
      & & & & & \ddots & \ddots & \ddots & \ddots & \ddots
    \end{array}
  \right]
  },
\end{equation*}
where the shaded block is the matrix $C$. Since the interior $n-2r$
rows of $D$ are known, the remaining unknown quantities lie within
$\tilde{P}$, $\hat{P}$, $B$ and $\hat{B}$.

\begin{remark}
  There are, naturally, a wide variety of modifications that can be
  made to this basic structure.  For example, there is no need to
  require that the left and right closures are identical in size and
  order, nor does the interior method need to be a centered finite
  difference method.  Moreover, the underlying computational grid need
  not be uniformly spaced~\cite{Mattsson2014}, and the a much more
  general SBP framework has been developed~\cite{Fernandez2014b}.
  Despite the wide variety of modifications available, however, it is
  still interesting to consider under what circumstances the
  conventional SBP first derivative operator can exist.
\end{remark}

\subsection{The question of existence}

The main question of this paper can now be stated as follows.
\begin{question}
  Let the positive integer triple $(s,t,r)$ be given.  Does there
  exist an SBP pair, $P$ and $D$, satisfying (P1)--(P3) such that $P$
  has the form given in~\eqref{eq:P-form} and $D$ has the form given
  in~\eqref{eq:D-form} with order of accuracy $t$ in its first $r$ and
  last $r$ rows and order of accuracy $2s$ in its remaining interior
  rows?  \qlabel{main}
\end{question}

We will say that an SBP operator for the triple $(s,t,r)$ exists if
the answer to \qref{main} is affirmative, and that no such SBP
operator exists if the answer is negative.  The principal
contributions of the present work are the following.
\begin{itemize}
\item We derive a set of compatibility conditions (similar to the
  conditions of Kreiss and Scherer~\cite{Kreiss1974}) on the triple
  $(s,t,r)$ that are necessary and sufficient for the existence of an
  SBP operator (\secref{compat}).  These conditions decouple the
  problem of constructing SBP operators into a two-step process: first
  the norm $P$ is constructed, if possible, and then the derivative
  matrix $D$ is constructed.  Compatibility conditions are given for
  the general block-norm setting (\thmref{full-solvability}) and are
  specialized to the diagonal-norm setting (\thmref{solvability}).
\item Focusing on the operators with diagonal norm, we describe a
  deterministic algorithm for answering the existence question
  (\secref{algorithm}).  The algorithm is quite simple, comprising the
  solution of a linear system of equations, followed by the solution
  of a standard linear program.
\item Next, we report the results of an automated search of the
  $(s,t,r)$-space, showing the existence of diagonal-norm SBP
  operators with orders of accuracy as large as $2s=30$ in the
  interior and $t=15$ on the boundary.  (The diagonal-norm operators
  of highest order accuracy in the literature are $2s=10$ and $t=5$.)
\item We follow this exploration with the description of a new
  algorithm for optimizing the SBP derivative operator (\secref{opt})
  and demonstrate the effectiveness of some newly constructed,
  high-order SBP operators by numerical experiment
  (\secref{numerics}).
\end{itemize}

\section{Compatibility conditions}
\seclabel{compat}

In this section, we present a derivation of necessary and sufficient
compatibility conditions (similar to those given in~\cite{Kreiss1974})
designed explicitly to allow for the cases $t<s$ and $r>2s$.  It is
sufficient to treat the top rows of $D$; the bottom rows are treated
analogously.  Consider the three matrices $X\in\R^{r\times(t+1)}$,
$\tilde{X}\in\R^{s\times(t+1)}$ and $Y\in\R^{r\times(t+1)}$ defined as
\begin{equation*}
  X_{ij} = i^j,\qquad
  \tilde{X}_{ij} = (r+i)^j,\qquad\text{and}\qquad
  Y_{ij} = ji^{j-1},
\end{equation*}
where, for convenience in dealing with the upper-left entries $X_{00}$
and $Y_{00}$ in what follows, we use the notation
\begin{equation*}
  0^0 = 1
  \qquad\text{and}\qquad
  0\cdot 0^{-1} = 0.
\end{equation*}
(This is only a notational convenience to avoid the need to consider
special cases in what follows.  These conventions are never treated as
computationally valid.)

 With these definitions, $D$ is a $t$th order
derivative approximation in the first $r$ rows if and only if
\begin{equation*}
  \tilde{P}^{-1}BX + \tilde{P}^{-1}C\tilde{X} = Y\quad\text{or equivalently}\quad
  BX + C\tilde{X} = \tilde{P}Y.
\end{equation*}
Moreover, $D$ satisfies the energy condition~\eqref{eq:energy} if and
only if
\begin{equation*}
  B + B^T = -e_0e_0^T,\qquad \hat{B}+\hat{B}^T = e_{r-1}e_{r-1}^T.
\end{equation*}
Splitting $B$ into its symmetric and antisymmetric parts yields
\begin{equation*}
  B = B_1 + B_2,\quad B_1 = \frac{1}{2}(B+B^T)=-\frac{1}{2}e_0e_0^T,\qquad B_2 = \frac{1}{2}(B-B^T) = -B_2^T.
\end{equation*}
Thus, the equation for $B$ and $P$ can be written as
\begin{equation}
  \eqlabel{general-system}
  \frac{1}{2}(B-B^T)X = B_2X = \tilde{P}Y-C\tilde{X}-B_1X.
\end{equation}

The form of \Eqref{general-system} suggests a solution strategy.
First, determine conditions on the norm $\tilde{P}$---the
compatibility conditions---which guarantee solvability of
\Eqref{general-system} for $B_2$.  If such a norm exists, then
necessarily a corresponding SBP derivative operator must exist.
Conversely, if \Eqref{general-system} is not solvable for any choice
of norm $\tilde{P}$, then we may conclude that no SBP operator/norm
pair exists for the given parameters.

Let $\cL:\R^{r\times r}\to\R^{r\times(t+1)}$ be the linear operator
with action $\cL B=\frac{1}{2}(BX-B^TX)$.  Then the standard
solvability theorem of linear algebra states that
\Eqref{general-system} can be solved if and only if the right-hand
side is orthogonal to the nullspace of $\cL^*$ (with orthogonality and
the adjoint defined in some inner product).  It is convenient to use
the standard inner product on $\R^{m\times n}$:
\begin{equation*}
  \left<U,V\right> = \trace(UV^T) = \sum_{i=0}^{m-1}\sum_{j=0}^{n-1}U_{ij}V_{ij}.
\end{equation*}
Keeping in mind two fundamental properties of the matrix trace, namely
that
\begin{equation*}
  \trace(U) = \trace(U^T)\qquad\text{and}\qquad\trace(UV)=\trace(VU)
\end{equation*}
whenever the dimensions of $U$ and $V$ are compatible with both
products, we can compute the adjoint of $\cL$ as follows.
\begin{equation*}
  \begin{split}
    2\left<\cL B,A\right> &= \trace(BXA^T)-\trace(B^TXA^T) =
    \trace(BXA^T)-\trace(AX^TB) \\&= \trace(BXA^T)-\trace(BAX^T) =
    \trace(B(AX^T-XA^T)^T) = \left<B,AX^T-XA^T\right>.
  \end{split}
\end{equation*}
So $\cL^*A = \frac{1}{2}(AX^T-XA^T)$.  Thus, we have proved the
following lemma.
\begin{lemma}
  \Eqref{general-system} is solvable if and only if there exists a
  $\tilde{P}$ such that $\left<\tilde{P}Y-C\tilde{X}-B_1X,A\right>=0$
  for every $A\in\R^{r\times(t+1)}$ satisfying $AX^T=XA^T$.
\end{lemma}

This lemma provides the first form of the compatibility condition.  An
SBP operator for the parameter triple $(s,t,r)$ exists if and only if
one can find a positive definite matrix $\tilde{P}$ such that the
right-hand side of \Eqref{general-system} is orthogonal to the
nullspace $N(\cL^*)$ of $\cL^*$, or equivalently, if and only if the
right-hand side is orthogonal to each element of a basis for
$N(\cL^*)$.  A basis is given by the following lemma.

\begin{lemma}
  A matrix $A$ satisfies $AX^T=XA^T$ if and only if $A=XE$ for some
  $E=E^T\in\mathbb{R}^{(t+1)\times(t+1)}$.
\end{lemma}
\begin{proof}
  Suppose $AX^T=XA^T$.  Since $X$ has full column rank, $X^TX$ is
  invertible, and, thus, 
  \begin{equation*}
    A = AX^TX(X^TX)^{-1} = XA^TX(X^TX)^{-1} = XE.
  \end{equation*}
  To see that $E$ is symmetric, observe that $XEX^T=AX^T=XA^T=XE^TX^T$
  and so $X^TXEX^TX = X^TXE^TX^TX$. Since $X^TX$ is invertible,
  $E=E^T$.  The converse is straightforward.
\end{proof}
\begin{corollary}
  \Eqref{general-system} is solvable if and only if there exists a
  $\tilde{P}$ such that
  \begin{equation}
    \eqlabel{FAT}
    \left<\tilde{P}Y-C\tilde{X}-B_1X,X(e_pe_q^T + e_qe_p^T)\right>=0
    \qquad\text{for all }\quad p = 0,1,\ldots,t,\quad q = p, p+1, \ldots, t.
  \end{equation}
\end{corollary}

Before forming the entire inner product in \Eqref{FAT}, it is useful
to compute a few separate terms. First, observe that for any matrix
$Z\in\mathbb{R}^{r\times(t+1)}$
\begin{equation*}
  \left<Z,Xe_pe_q^T\right> = \trace(Ze_q(Xe_p)^T) = \trace((Xe_p)^TZe_q)
  = (Xe_p)^TZe_q.
\end{equation*}
Now, considering the first term in the inner product in \Eqref{FAT},
we find
\begin{equation*}
  \left<\tilde{P}Y,Xe_pe_q^T\right> = (Xe_p)^T(\tilde{P}Ye_q)
  = \sum_{k=0}^{r-1}X_{kp}\sum_{\ell=0}^{r-1}\tilde{P}_{k\ell}Y_{\ell q}
  = \sum_{k=0}^{r-1}\sum_{\ell=0}^{r-1}qk^p\ell^{q-1}\tilde{P}_{k\ell}.
\end{equation*}
Turning to the second form in the right-hand side, we find
\begin{equation*}
  \left<C\tilde{X},Xe_pe_q^T\right> = (Xe_p)^T(C\tilde{X}e_q)
  = \sum_{k=0}^{r-1}X_{kp}\sum_{\ell=0}^{s-1}C_{k\ell}\tilde{X}_{\ell q}
  = \sum_{k=0}^{r-1}\sum_{\ell=0}^{s-1}k^p(r+\ell)^qC_{k\ell}.
\end{equation*}
Finally, we find
\begin{equation*}
  \left<B_1X,Xe_pe_q^T\right> = (Xe_p)^T(B_1Xe_q)
  = \sum_{k=0}^{r-1}X_{kp}\sum_{\ell=0}^{r-1}(B_1)_{k\ell}X_{\ell q}
  = -\frac{1}{2}\sum_{k=0}^{r-1}\sum_{\ell=0}^{r-1}k^p\ell^q\delta_{k0}\delta_{\ell 0}
  = -\frac{1}{2}\delta_{p0}\delta_{q0}.
\end{equation*}
Combining the above computations, we arrive at a system of $(t+1)(t+2)/2$
equations for $\tilde{P}$.
\begin{theorem}
  \thmlabel{full-solvability} An SBP operator with parameters
  $(s,t,r)$ exists if and only if there is a positive definite matrix
  $\tilde{P}$ satisfying
  \begin{equation}
    \eqlabel{full-p-system}
    \sum_{k=0}^{r-1}\sum_{\ell=0}^{r-1}\left(qk^p\ell^{q-1}+pk^q\ell^{p-1}\right)\tilde{P}_{k\ell}
    = \sum_{k=0}^{r-1}\sum_{\ell=0}^{s-1}\left(k^p(r+\ell)^q+k^q(r+\ell)^p\right)C_{k\ell}
    - \delta_{p0}\delta_{q0}
  \end{equation}
  for every $p=0,1,\ldots,t$ and $q=p,p+1,\ldots,t$.
\end{theorem}

\subsection{SBP operators with diagonal norm}

The remainder of the paper is restricted to the case that $P_h$ is
diagonal, which is quite natural due to the fact that these are the
only SBP operators for which standard techniques exist for proving
stability for PDEs with variable coefficients or on multi-dimensional
curvilinear grids~\cite{Svard2004}.  Although SBP operators with
non-diagonal (block) norm have recently been shown to be stabilizable
on curvilinear grids~\cite{Mattsson2013} by the addition of a tuned
artificial damping term, the question of existence of diagonal-norm
SBP operators remains an interesting open problem; no such operators
with interior order greater than 10 exist in the literature.  Under
the assumption that $P_h$ is diagonal, \thmref{full-solvability}
simplifies somewhat.
\begin{theorem}
  \thmlabel{solvability} An SBP operator with parameters $(s,t,r)$ and
  diagonal $P_h$ exists if and only if there is a diagonal positive
  definite matrix $\tilde{P}$ satisfying
  \begin{equation}
    \eqlabel{p-system}
    \sum_{k=0}^{r-1}(p+q)k^{p+q-1}\tilde{P}_{kk}
    = \sum_{k=0}^{r-1}\sum_{\ell=0}^{s-1}\left(k^p(r+\ell)^q+k^q(r+\ell)^p\right)C_{k\ell}
    - \delta_{p0}\delta_{q0}
  \end{equation}
  for every $p=0,1,\ldots,t$ and $q=p,p+1,\ldots,t$.
\end{theorem}
\begin{remark}
  Since the left-hand side of \Eqref{p-system} is zero when $p=q=0$,
  the SBP operator can only exist if $\sum_{k,\ell}C_{k\ell} =
  \frac{1}{2}$.  If a generic row of the centered difference portion
  of $D$ has the coefficients
  \begin{equation*}
    \begin{bmatrix}
      \cdots & -\alpha_s & \cdots & \alpha_{-1} & 0 & \alpha_1 & \cdots & \alpha_s & \cdots
    \end{bmatrix}
  \end{equation*}
  then
  \begin{equation*}
    \sum_{k,\ell}C_{k\ell} = \sum_{i=1}^s i\alpha_i,
  \end{equation*}
  so the $p=q=0$ equation is simply the requirement that the centered
  difference operator evaluate derivatives of linear functions
  exactly, which will always be true in the present setting.  The
  remaining $t(t+3)/2$ equations do involve the $r$ unknowns in
  $\tilde{P}$.
\end{remark}

\begin{remark}
  \label{rem:square}
  Although \Eqref{p-system} appears to be an overdetermined system if
  $r<t(t+3)/2$, it is known (see \cite[Theorem 2.1]{Kreiss1974}) that,
  when $r=s=2t$, \Eqref{p-system} has a unique solution.  Moreover,
  by~\cite[Corollary 1]{Hicken2013}, if a norm $P$ exists, then the
  system must consistent with the requirement that the norm matrix $P$
  act as a $2s$-order quadrature rule, providing a lower bound on the
  number of linearly independent equations that must be present for
  solvability.  In this paper, we will not concern ourselves with
  locating the linearly independent equations since \Eqref{p-system}
  is sufficient for our purposes.
\end{remark}

\begin{remark}
  \label{rem:non-square}
  It is important that $r$ be allowed to vary independently of $s$,
  since it appears (see \secref{smallest-r}) that if $s=t\ge 5$ and
  $r=2s$, then the unique solution to \Eqref{p-system} is not positive
  and, therefore, that in general SBP operators do not exist for
  $(s,t,r)=(s,s,2s)$.
\end{remark}

\section{Algorithm for existence}
\seclabel{algorithm}

This section describes a method for algorithmically deciding the
answer to \qref{main} for given parameters $(s,t,r)$.  That is, the
algorithm presented here determines whether or not such an
operator/norm pair exists, but does not completely construct one.  The
construction, assuming existence is known, is postponed until
\secref{opt}.  The entire process that follows is performed in exact
arithmetic (i.e., using rational numbers) since the solvability of
\Eqref{general-system} requires that $\tilde{P}$ be an exact solution
to \Eqref{p-system}.  If $\tilde{P}$ is only an approximate solution,
then \Eqref{general-system} is not solvable.  Of course, there should
be approximate solutions to \Eqref{general-system} in this case, but
the details of a finite precision implementation of this algorithm
remains an open question (see \secref{conclusion}).  For the results
presented in this paper, the Python library
\texttt{sympy}~\cite{sympy} was used for exact arithmetic.

\subsection{Solve the linear system} To initialize the algorithm,
\Eqref{p-system} is put into the form of a matrix-vector equation
$Ax=b$ with $A$ a $[(t+1)(t+2)/2]\times r$ matrix and $x$ the vector
of unknowns $\tilde{P}_{kk}$.  This can be done in exact arithmetic by
Gaussian elimination.  As stated in Remarks~\ref{rem:square}
and~\ref{rem:non-square}, if $r=2s=2t$, \Eqref{p-system} has a unique
solution.  In this case, the answer to \qref{main} is immediate.  If
$x$ is positive, an SBP operator exists for the given parameters.  If
$x$ has any non-positive entries, no SBP operator exists.  This is
exactly Theorem~2.1 of~\cite{Kreiss1974}.  For other choices of
$(s,t,r)$, however, there is typically a solution manifold of the form
\begin{equation*}
  x = x_0 + Gy,
\end{equation*}
where $G$ is $r\times v$, with $v$ the number of degrees of freedom in
the solution.  In this case, more work is required to determine if
there is a solution with positive entries.

\subsection{Solve the LP problem}
\seclabel{lp}

If the previous steps produced a manifold of solutions to the linear
system in \Eqref{p-system}, \qref{main} is equivalent to asking whether
there exists a $y\in\R^v$ such that $x_0+Gy$ has all positive entries.
To see how this problem can be solved algorithmically, consider the
following optimization problem
\begin{equation}
  \eqlabel{opt}
  \begin{split}
    \text{maximize}   &\qquad\min_ix_i, \\
    \text{subject to} &\qquad x = x_0 + Gy,
  \end{split}
\end{equation}
and note that \Eqref{p-system} has a strictly positive solution if and
only if the value of the optimization problem is strictly positive.
The optimization problem can be algorithmically solved through a
common technique that transforms the problem into a standard linear
program (LP):
\begin{equation}
  \eqlabel{lp}
  \begin{split}
    \text{maximize}   &\qquad \eta, \\
    \text{subject to} &\qquad x_i\ge\eta,\quad i=0,1,\ldots,r-1, \\
    &\qquad x = x_0 + Gy.
  \end{split}  
\end{equation}
As an LP, \Eqref{lp} can be solved by the simplex method, thus
providing an algorithm for solving \Eqref{opt}.  In this case, we
conclude that the SBP operator exists if and only if the solution to
\Eqref{lp} is positive.  From an implementation perspective, this is
the most complex step, as it requires a simplex solver in exact
arithmetic.  We did not find an existing library for this, and so
implemented our own simplex solver in Python.

\begin{remark}
  \label{rem:choice-of-P}
  By openness, it is clear that if a manifold of solutions exists and
  if there is one positive solution, then there are infinitely many
  positive solutions.  In this case, it is not clear which choice of
  $P$ is ``best'' in any particular sense.  In this paper, we choose
  the $P$ maximizing \Eqref{opt} as a particular choice.  This is not
  quite arbitrary, as described in Remark~\ref{rem:best-P-norm}.
\end{remark}

\section{Existence and nonexistence of SBP operators}
\seclabel{existence}

\begin{table}
  \centering
  \begin{tabular}{cccccl}
    $s$ & $t$ & $r$ & dof $P_h$  & dof $D_h$ & $\min_ix_i$ \\
    \hline
    1 & 1 & 1 & 0 & 0 & 5.000e-01 \\
    2 & 2 & 4 & 0 & 0 & 3.541e-01 \\
    3 & 3 & 6 & 0 & 1 & 3.159e-01 \\
    4 & 4 & 8 & 0 & 3 & 2.575e-01 \\
    5 & 5 & 11 & 1 & 10 & 2.077e-01 \\
    6 & 6 & 14 & 2 & 21 & 9.683e-03 \\
    7 & 7 & 19 & 5 & 55 & 1.907e-01 \\
    8 & 8 & 23 & 7 & 91 & 4.652e-02  \\
    9 & 9 & 28 & 10 & --- & 4.622e-02 \\
    10 & 10 & 34 & 14 & --- & 8.357e-02\\
    11 & 11 & 40 & 18 & --- & 3.907e-02\\
    12 & 12 & 47 & 23 & --- & 5.286e-02\\
    13 & 13 & 54 & 28 & --- & 1.933e-02\\
    14 & 14 & 62 & 34 & --- & 1.863e-02\\
    15 & 15 & 71 & 41 & --- & 4.559e-02 \\
  \end{tabular}
  \caption{For the case $t=s$, the table reports the \emph{smallest} value of 
    $r$ for which an SBP operator exists.  When $P_h$ is non-unique, 
    the number of degrees of freedom in $P_h$ is reported as 
    ``dof $P_h$''.  The
    value of \Eqref{opt} is reported as $\min_ix_i$.  For the optimal 
    $P_h$, the degrees of freedom of $D_h$ is reported as
    ``dof $D_h$''.  Since the computation of the solution space in $D_h$ 
    is much more computationally taxing than that of $P_h$, only the 
    cases with $s\le 8$ include the dimension of the solution space for 
    $D_h$.}
  \tablabel{s-results}
\end{table}

This section presents some novel results based on the algorithm of the
previous section.  It is worth remarking that, although the algorithm
is provably correct, the following computational results rely on
computer-assisted proof.  The numerators and denominators of the
rational numbers involved are sufficiently large that we cannot hope
to perform the Gaussian elimination and simplex method steps by hand
except in a small number of cases.  For example, the value for $x_0$
in the case $s=t=8$, $r=23$ is
\begin{equation*}
  x_0=\frac{83852077150009258297147}{299027329581685985280000}.
\end{equation*}
Although we have done our best to test our code and to compare with
SBP results in the literature, the results presented in this text are
nevertheless vulnerable to errors either in the \texttt{sympy}
rational number manipulation routines, the Gaussian elimination
routine or in the simplex solver.  As an example, an earlier version
of the code produced incorrect results from time to time due to some
unexpected behavior in the symbolic operations of a particular
commercial software tool.  We have a high degree of confidence in the
computational results presented in this paper, but certainly encourage
their verification by others.

\subsection{The smallest $r$ for $t=s$} 
\seclabel{smallest-r}

We first consider the case of an SBP operator of order $2s$ in the
interior and $t=s$ on the boundary.  It can be readily seen that if a
solution exists for a particular choice of $(s,t,r)$, then this is
also a solution for $(s,t,r')$ for any $r'>r$.  That is, if the answer
to \qref{main} is affirmative for a particular choice of parameters,
the answer is also affirmative if the finite difference orders are
left unchanged and the boundary closure size is increased.  Thus, it
is possible to perform a bisection search to locate the smallest $r$
for which an SBP operator with the given choice $s=t$ exists.
\tabref{s-results} presents the results of this parameter sweep.  The
table also gives the dimension of the solution space of $P_h$ and the
value of \Eqref{opt}.

As stated previously, the present discussion does not concern the
actual construction of an SBP operator, but merely the question of
existence.  However, once a suitable norm $P_h$ is found,
\Eqref{general-system} is then guaranteed solvable for the block $B$
of an SBP difference operator $D_h$.  Included in \tabref{s-results}
is a column titled ``dof $D_h$'', which reports (for some choices of
$s$) the dimension of the solution set for $D_h$ with the choice of
$P_h$ described in Remark~\ref{rem:choice-of-P}.  In
\secref{opt}, we discuss a method for choosing a particular $D_h$
from this solution set.

To our knowledge, the results for $s>5$ are unknown in the literature.
Of particular interest is the nonlinear dependence of $r$ on $s$ (see
Remark~\ref{rem:non-square}).  Based on previous results for the cases
$s=2$, $3$ and $4$, it might be expected that SBP operators exist for
all $r=2s=2t$ (the case considered in~\cite{Kreiss1974}).  However, by
the nature of the parameter sweep conducted here, we conclude that,
for a given $s$, no SBP operator (with $t=s$) exists for $r$ smaller
than the value given in the table.  The case $s=5$, $r=11$ has already
been reported (see, e.g., \rrefs{Diener2007,Mattsson2013}).  However,
while a footnote of \rref{Diener2007} states that the choice $s=5$,
$r=10$ ``did not result in a positive definite norm'', no proof or
explanation is given.

\subsection{The largest $t$ for $r=2s$}

\begin{table}
  \centering
  \begin{tabular}{ccccl}
    $s$ & $t$ & $r$ & dof $P_h$  & $\min_ix_i$ \\
    \hline
    1 & 1 & 2  & 0 & 5.000e-01 \\
    2 & 2 & 4  & 0 & 3.542e-01 \\
    3 & 3 & 6  & 0 & 3.159e-01 \\
    4 & 4 & 8  & 0 & 2.575e-01 \\
    5 & 4 & 10  & 2 & 3.367e-01 \\
    6 & 5 & 12  & 2 & 2.997e-01 \\
    7 & 6 & 14  & 2 & 9.682e-03 \\
    8 & 6 & 16  & 4 & 2.992e-01 \\
    9 & 6 & 18  & 6 & 3.207e-01 \\
    10 & 7 & 20  & 6 & 2.923e-01 \\
    11 & 7 & 22  & 8 & 3.088e-01 \\
    12 & 8 & 24  & 8 & 2.504e-01 \\
    13 & 8 & 26  & 10 & 2.980e-01 \\
    14 & 9 & 28  & 10 & 4.622e-02 \\
    15 & 9 & 30  & 12 & 2.858e-01 \\
  \end{tabular}
  \caption{For the case $r=2s$, the table reports the \emph{largest} value of 
    $t$ for which an SBP operator exists.  When $P_h$ is non-unique, 
    the number of degrees of freedom in $P_h$ is reported as 
    ``dof $P_h$''.  The
    value of \Eqref{opt} is reported as $\min_ix_i$.}
  \tablabel{t-not-s}
\end{table}

From the previous results, it is clear that asking for $r=2s=2t$ is,
in general, too restrictive and, if we require that $D_h$ have its
boundary order of accuracy half as large as its interior order of
accuracy, the size of the boundary closure must grow faster than $2s$.
As a second application of the algorithm, we consider the opposite
question.  Suppose we wish to have $r=2s$.  What is the largest
boundary order of accuracy $t$ for which an SBP operator exists?
\tabref{t-not-s} presents the results of such a parameter study.
Evidently, it is not difficult to find SBP operators with $r=2s$,
provided $t$ is allowed to grow more slowly than $s$.

\section{Optimization of the derivative operator}
\seclabel{opt}

\begin{figure}
  \centering
  \includegraphics[width=0.45\textwidth]{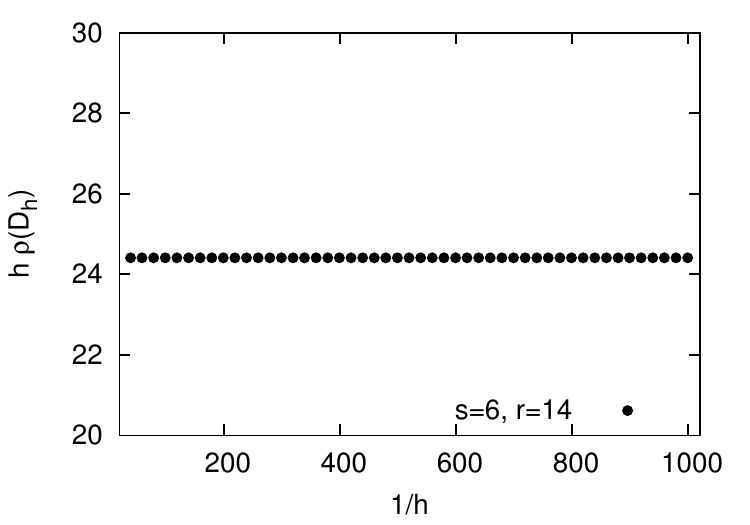}
  \includegraphics[width=0.45\textwidth]{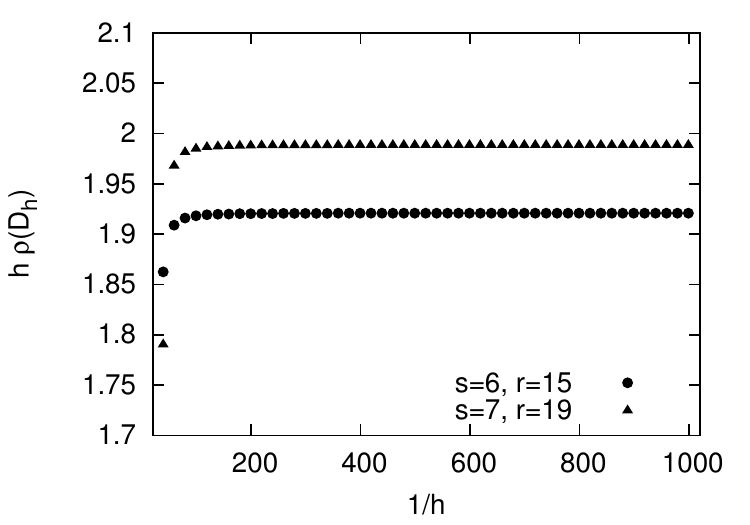}
  \caption{On the left, the spectral radius as a function of $1/h$ for
    the optimized $s=t=6, r=14$ SBP operator.  On the right, the
    spectral radii as functions of $1/h$ for the optimized $s=t=6,
    r=15$ and $s=t=7, r=19$ SBP operators.}
  \figlabel{spec-rad}
\end{figure}

Another interesting observation about the results presented in
\tabref{s-results} is that the number of degrees of freedom in $D_h$
grows rapidly with increasing $s$.  At the end of the algorithm
described in \secref{algorithm} we are generally left, not with a
single SBP operator, but with an entire linear manifold:
\begin{equation}
  \eqlabel{D-manifold}
  D_h = D_0 + \sum_j\xi_jD_j,
\end{equation}
where $j$ varies through all degrees of freedom.  Although the primary
concern of the present paper is the \emph{existence} of SBP operators,
the question of ``Which is best?''  is also important.  As described
in \rref{Diener2007}, there are a wide range of options for defining
``best''.  Rather than consider each of these, we focus on a
particular objective function: minimizing the spectral radius
$\rho(D_h)$.  This objective function is interesting because it
controls the CFL condition for explicit PDEs solvers.  Moreover, it is
unique among the objective functions considered in the reference in
that it is non-convex in the parameters $\xi_j$ of $D_h$, and
therefore difficult to optimize globally.  Hence the remark in
\rref{Diener2007}: ``Therefore, when we refer to minimizing the
spectral radius, we perform a numerical minimization and do not claim
that we have actually found a global minimum.''

In this paper, we suggest an alternative to minimizing the spectral
radius directly.  The key point is that, although $\rho(C)$ is not a
convex function in the entries of $C$ in general, it \emph{is} convex
for normal matrices $C$.  And, although $D_h$ is not a normal matrix
in general, it is close in some sense to a normal matrix, because
\begin{equation*}
  \left(P_hD_h - \frac{1}{2}Q\right) + \left(P_hD_h - \frac{1}{2}Q\right)^T
  = 0
\end{equation*}
(see \Eqref{energy}).  Since the surrogate matrix $P_hD_h -
\frac{1}{2}Q$ is skew-symmetric and therefore normal, its spectral
radius agrees with its operator $2$-norm and, thus, is a convex
function of its entries.  Moreover, in this norm
\begin{equation*}
  \rho(D_h) \le \|D_h\| \le \|P_h^{-1}\|\cdot\|P_hD_h\| \le
  \|P_h^{-1}\|\cdot\left(\|P_hD_h-\frac{1}{2}Q\| + \frac{1}{2}\right).
\end{equation*}
Provided $\|P_h\|$ is not too large, minimizing the norm of the
surrogate matrix tends to make the spectral radius of $D_h$ small.
Defining
\begin{equation*}
  C_0 = P_hD_0 - \frac{1}{2}Q,\qquad C_i = P_hD_i,\quad\text{and}\quad
  C(\xi) = C_0 + \sum_j\xi_jC_j,
\end{equation*}
the goal is to minimize $\|C(\xi)\|$ with respect to $\xi=(\xi_j)$.
It turns out that this problem can be easily transformed into a
Semidefinite Program (SDP)~\cite[Sec.~4.6.3]{Boyd2004}, treatable by a
number of standard solvers.

Our implementation of this idea is to choose a particular $N$ (we
chose $N=100$ for our examples) and to numerically minimize the norm
of the surrogate $\|C(\xi)\|$ with respect to $\xi$.  Unlike in the
previous section, there is no apparent need to perform this
optimization in exact arithmetic.  Instead, the elements of the $C_i$
are evaluated in double precision and are used to set up the SDP,
which is then solved through the \texttt{cvxopt}
package~\cite{cvxopt}.  Once the optimal $\xi$ is found, the
corresponding $D_h$ is formed from \Eqref{D-manifold}.

Using this technique on the case $s=t=7$, $r=19$ (with a
55-dimensional search space) we located an SBP operator such that
$\rho(D_h) \approx 2/h$, as verified with several choices of $h$.
Applying the same technique to the case $s=t=6$, $r=14$, however, did
not produce suitable results, as might be expected from careful
inspection of \tabref{s-results}.  In particular, the value
$\min_ix_i$ associated with this case is a very small number, implying
that $\|P_h^{-1}\|$ is large.  When we attempted to minimize the
surrogate matrix in this case the resulting SBP operator exhibited
(numerically) the scaling $\rho(D_h) \approx 24/h$---significantly
larger than one might wish.  This problem can be remedied as follows.

Recall that if the SBP equations are solvable for $(s,t,r)$, then they
are solvable for $(s,t,r')$ for any $r'>r$.  In general, increasing
$r$ leads to a larger number of degrees of freedom in both $P_h$ and
$D_h$.  So, it is reasonable to ask whether choosing $r>14$ might
improve the result in the case $s=6$.  With the choice $s=6$, $r=15$
there is an SBP operator with $\min_ix_i\approx 0.24$.  In this case,
the resulting SBP operator exhibits the scaling $\rho(D_h) \approx
1.92/h$.  Thus, we conclude, that the smallest $r$ for which an SBP
operator exists is not necessarily the best choice of $r$.
Apparently, it is useful to choose an $r$ for which $\|P_h^{-1}\|$ is
not too large.  

\begin{remark}
  \label{rem:best-P-norm}
  It is interesting to note that, although the objective function in
  \Eqref{opt} was not chosen specifically for this property, a
  side-effect of the algorithm described in this paper is to choose,
  among all possible $P_h$, the one with smallest $\|P_h^{-1}\|$.
\end{remark}

\figref{spec-rad} shows the numerically computed spectral radii of the
operators described in this section as functions of $1/h$.  The
coefficients for the optimized $s=6$ ($r=15$) and $s=7$ operators are
included online as supplementary data for this paper, as described in
\appref{coeffs}.

\section{Numerical experiments}\seclabel{numerics}

\begin{figure}
  \centering
  \includegraphics[width=0.48\textwidth]{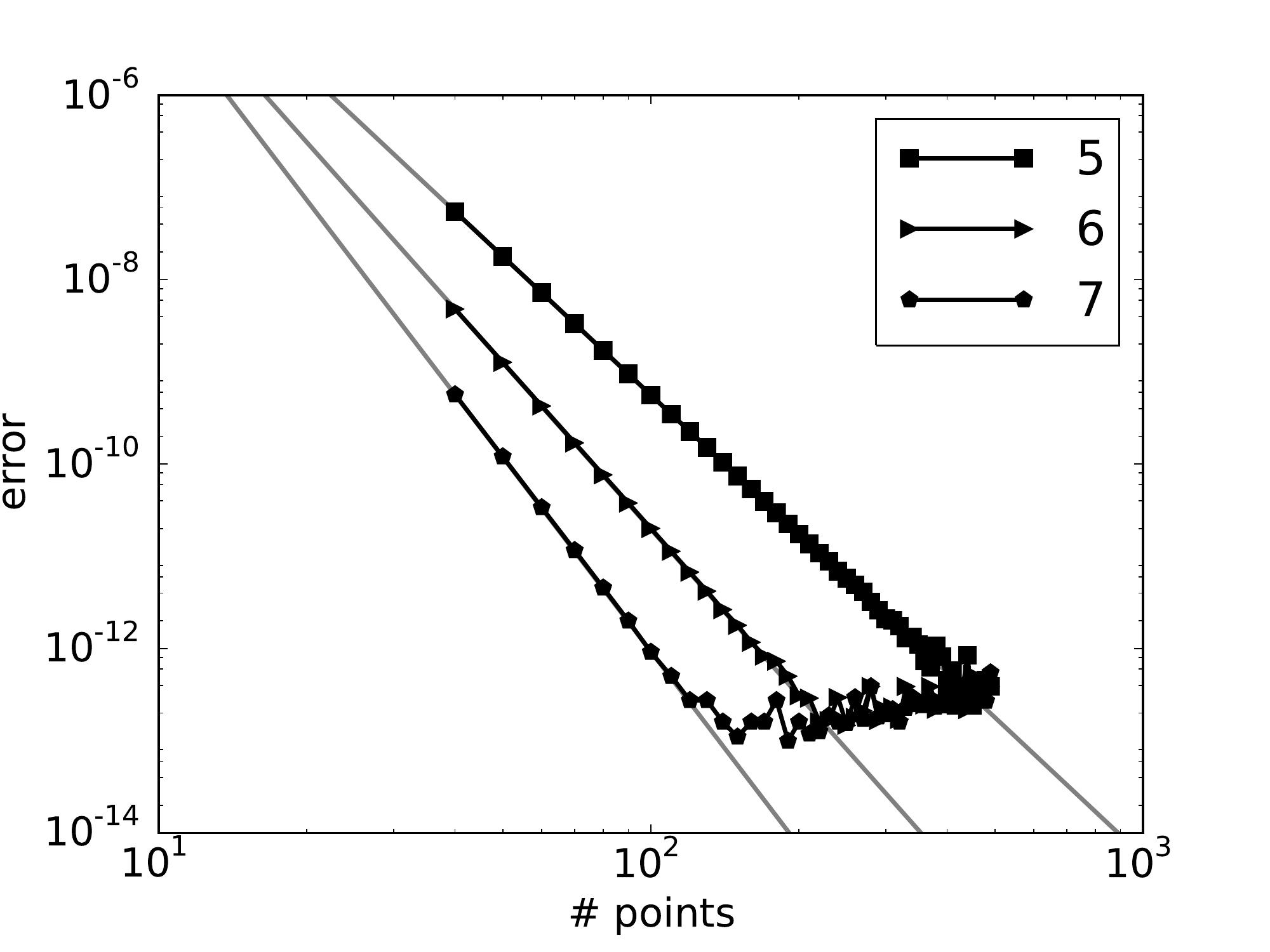}
  \caption{The $\ell^\infty$ error in approximating the first
    derivative of the function $f(x)=e^x$ on the interval $[0,1]$
    plotted against the number of sample points for several SBP
    operators suggesting the appropriate orders of convergence.  The
    integer values in the legend refer to the boundary order ($t=s$)
    of the SBP operator.  The gray lines denote $5$th, $6$th, and
    $7$th order slopes respectively.}
  \figlabel{numerics-convergence}
\end{figure}

As a first test of the SBP operators, we consider a simple convergence
study.  Taking $N$ uniformly spaced sample values of the function
$x\mapsto e^x$ on the interval $[0,1]$, we compare the numerically
computed derivative to the true derivative in $\ell^\infty$ norm.
\figref{numerics-convergence} shows the results of this study.  The
error is plotted against the number of grid points for three SBP
operators.  The integer values ($5$, $6$, and $7$) in the legend are
the values of $s$ for each operator.  All operators have boundary
order $t=s$ and interior order $2s$.  The width of the closure is
chosen to be as small as possible (see~\tabref{s-results}) except in
the case $t=s=6$, in which case $r=15$ is chosen as described in
\secref{opt}.  Note in particular that $P$ is selected as described in
\secref{lp}.  When the SBP operator is not unique, $D$ is selected via
the optimization problem described in \secref{opt}.  The results in
\figref{numerics-convergence} demonstrate that the constructed
operators exhibit the appropriate orders of accuracy.

A more interesting test is to consider the operators' performance in
solving the advection PDE
\begin{equation*}
  u_t + u_x = 0\qquad x\in(0,1000),\quad t>0
\end{equation*}
with initial and boundary conditions
\begin{equation*}
  u(x,0) = 0,\qquad u(0,t) = g(t) = \exp\left( -a(t+10)^2 \right)
\end{equation*}
where
\begin{equation*}
  a = \frac{-\log\left(10^{-16}\right)}{100}.
\end{equation*}
The weight $a$ in the Gaussian is designed so that the function is
essentially supported (within a tolerance of $10^{-16}$) in an
interval of width $20$.  When the PDE is solved to time $t=1000$, the
Gaussian hump moves in from the left boundary and traverses the domain
to end centered over the point $x=990$ (essentially supported in the
interval $[980,1000]$ (see \figref{compare-advect}).

\begin{figure}
  \centering
  \includegraphics[width=0.48\textwidth]{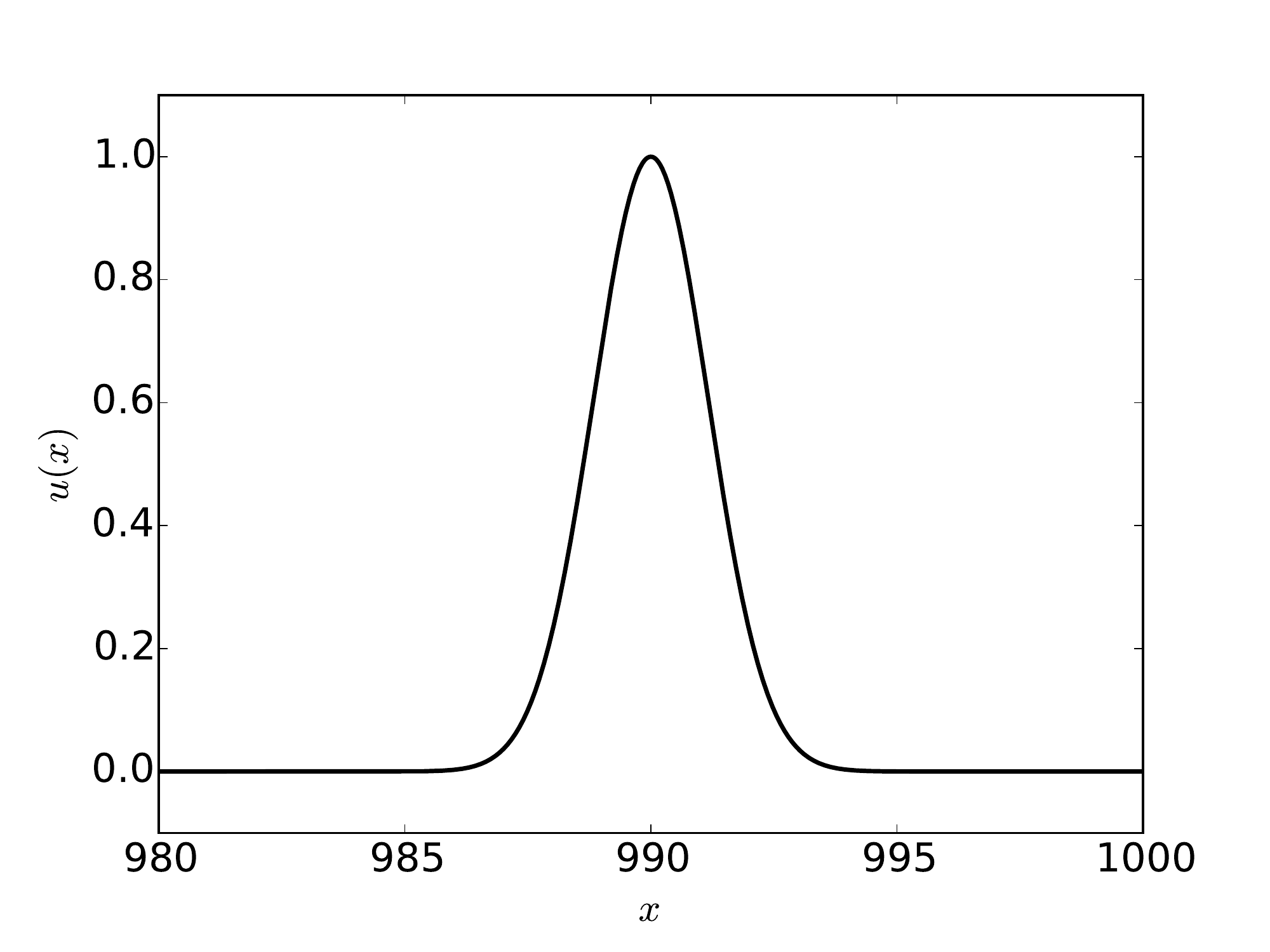}  
  \includegraphics[width=0.48\textwidth]{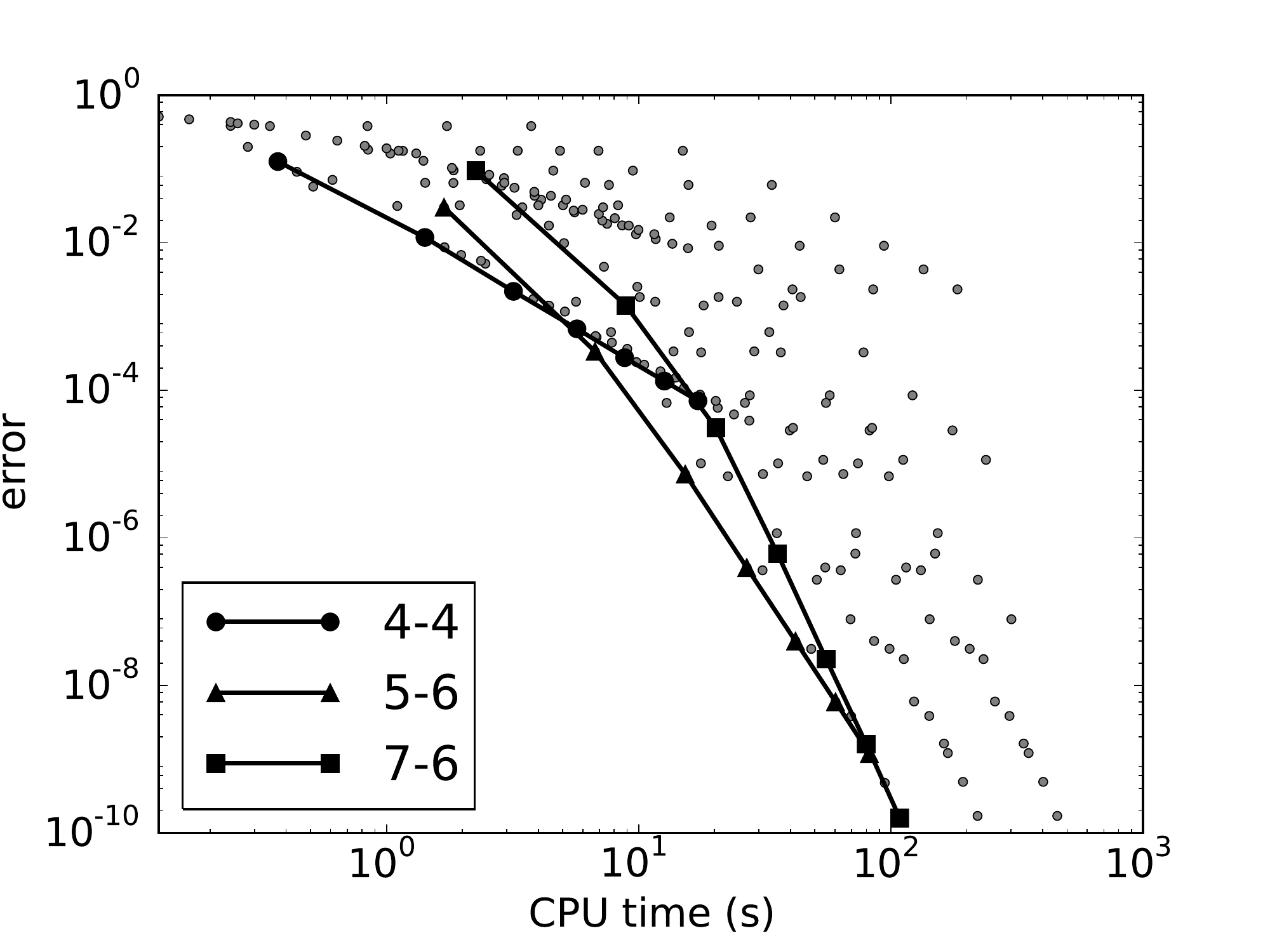}
  \caption{(left) Solution to the advection example of
    \secref{numerics} at time $t=1000$. (right) Error as a function of
    CPU time for the advection example using a variety of SBP
    operators and Adams-Bashforth methods.  Points labeled in the form
    $s$-$q$ indicate the results using an SBP operator with interior
    order $2s$ and Adams-Bashforth time integration of order $q$.}
  \figlabel{compare-advect}
\end{figure}

The purpose of the present study is to explore the best choice of
method for solving the problem.  The answer, of course, depends on the
definition of \emph{best} and also on the methods available for
comparison.  For this particular study, \emph{best} shall mean that
the method achieves a given accuracy in the shortest amount of CPU
time.  In this case, the CPU is a 1.6GHz desktop Intel Xeon processor
running an advection solver written in Fortran.  Spatial derivatives
were approximated using SPB operators of interior order from $4$ to
$14$ and temporal integration was performed by $q$th order
Adams-Bashforth (AB$q$) with $q\in\{3,4,6,7,8\}$.  (Second- and
fifth-order AB are excluded from consideration because their stability
regions include an insufficient amount of the imaginary axis.)  These
options are summarized in \tabref{CFL}.  As in the previous test, we
selected SBP operators with $t=s$ and with $r$ as small as possible
except for the $s=6$ case, with $P$ and $D$ generated as described in
Sections~\ref{sec:lp} and~\ref{sec:opt} respectively.  The CFL scaling
used to ensure stability is given by the product of cfl$_{1}$ and
cfl$_{2}$ given in the table, with the time step size $k$ and spatial
step size $h$ related via cfl$_{1}\cdot$cfl$_{2}k = h$.  The boundary
conditions were enforced using a simultaneous approximation term (SAT)
approach~\cite{Carpenter1994}, with the semidiscrete equations taking
the form
\begin{equation*}
  v_t + D_hv = -(v^Te_0-g(t))P_h^{-1}e_0.
\end{equation*}

Asymptotically, with unlimited precision, it is clear that a
higher-order method must always outperform a lower-order one.
However, for any given problem and a given meaningful range of
discretization step sizes in finite precision, it is not clear which
choice of spatial or temporal order will achieve a given accuracy in
the shortest time.  In order to compare the methods, we performed a
series of computations using each choice of SBP operator paired with
each AB integrator, each for the choices
$N\in\{2000,4000,6000,8000,10000,12000,14000\}$ grid points on the
interval $[0,1000]$--giving a total of $6\times 5\times 7=210$
computations.  In each computation, the advection problem was solved
to time $t=1000$, and the maximum absolute error (computed on the
computational grid) at this final time was recorded.

The results of the tests are shown in a scatter plot in
\figref{compare-advect}.  To simplify the presentation, we have chosen
to explicitly identify only three particular methods in the cloud of
points.  These methods are labeled with the form $s$-$q$ in the
legend, referring to the combination of AB$q$ and the SBP operator
with interior order of accuracy $2s$.  For this configuration, the SBP
operator with $s=t=4$ paired with AB4 is most effective for errors
down to approximately $10^{-4}$.  If smaller errors are required, the
$s=t=5$ SBP operator paired with AB6 performs well until approximately
the error level $10^{-10}$, where it is overtaken by the $s=t=7$ SBP
operator paired with AB6.  

Although it is impossible to generalize the results of these tests to
predictions for PDEs solvers in general, the outcome does illustrate a
general principle suggested by common sense---there is no single
\emph{best} numerical method for all problems.  Unsurprisingly, if
moderate errors are tolerable, a lower-order method does better; the
CFL conditions are less stringent than for high-order methods and the
constants preceding the leading-order error terms tend to be smaller.
As the error tolerance decreases, however, the asymptotic accuracy of
the higher-order methods start to win out.  It seems a difficult task
to determine \emph{a priori} which spatial and temporal order of
accuracy will lead to the desired results with the least expenditure
of computational time.  However, as the example illustrates, there may
be realistic instances when very high-order SBP operators are more
efficient than lower-order ones.

\begin{table}
  \centering
  \begin{tabular}[t]{c|c|c|c}
    $s$ & $t$ & $r$ & cfl$_1$\\
    \hline
    2 & 2 & 4  & 1.4\\
    3 & 3 & 6  & 1.6\\
    4 & 4 & 8  & 1.8\\
    5 & 5 & 11 & 1.9\\
    6 & 6 & 15 & 2.0\\
    7 & 7 & 19 & 2.1
  \end{tabular}
  \hspace{3em}
  \begin{tabular}[t]{c|c}
    $q$ &  cfl$_2$\\
    \hline
    3 & 1.39 \\
    4 & 2.38 \\
    6 & 8.93 \\
    7 & 17.5 \\
    8 & 34.1 \\
  \end{tabular}
  \caption{SBP parameters (left) and Adams-Bashforth order (right) versus CFL multipliers for the advection example in \secref{numerics}.  For SBP parameters $(s,t,r)$ and Adams-Bashforth order $q$, the time step was selected so that cfl$_{1}\cdot$cfl$_{2}k = h$.}
  \tablabel{CFL}
\end{table}

\section{Conclusion and future research}
\seclabel{conclusion}

This paper introduces an algorithm that provably answers the question
of existence of diagonal-norm SBP operators parameterized by a triple
$(s,t,r)$, and demonstrates the need to move away from the standard
choice of $(s,t,r)=(s,s,2s)$ when $s\ge 5$.  Our hope is that this
approach will lead to further research in the field, and to this end,
we conclude with a list of what we consider to be interesting
directions of future research.

\paragraph{Floating point algorithms} As remarked in
\secref{algorithm}, the current algorithm relies crucially on the use
of exact arithmetic.  This appears to be the principal bottleneck
preventing the discovery of even higher-order SBP operators than those
presented in this text, since this representation leads to very large
memory requirements and computationally expensive arithmetic
operations.  It would be interesting to know if there exists a similar
algorithm for finding (approximate) SBP operators in floating point
arithmetic.  Even the use of a variable-precision library would be an
improvement over the need for rational numbers.

\paragraph{Alternative energy norms} The algorithm described in
\secref{lp} chooses, among all possible $P_h$, the one that maximizes
$\min_iP_{ii}$.  This choice is useful for two reasons.  First, if the
optimal value is non-positive, then we immediately conclude that no
SBP operator exists.  Moreover, as remarked in \secref{opt}, this
choice is good for the application of optimizing the spectral radius
of the SBP derivative operator.  On the other hand, if
\Eqref{p-system} has a manifold of solutions and if one of these
solutions is strictly positive, then, in fact, \Eqref{p-system} has an
infinite number of strictly positive solutions.  It is not clear how
the choice of $P_h$ might influence later steps of the algorithm.

\paragraph{Compact stencil sizes}  As can be seen from
\tabref{s-results}, the closure size, $r$, appears to grow rapidly as
$s$ increases.  For example, the $18$th order ($s=9$) centered
operator requires at least $28$ points for a $9$th order closure.
This derivative operator contains a $9$th order derivative
approximation with a stencil width of $28+9=37$.  It is not clear
whether a method with such a wide stencil would actually be useful in
applications, and it would be interesting to know if there are
generalizations of the SBP framework that can reduce this size.

\section*{Acknowledgments}
The authors are deeply indebted to Daniel Appel\"{o}, Jeremy Kozdon
and Anders Petersson for their feedback and suggested improvements on
early drafts of this manuscript.  Joshua Klarmann's work on this
project was sponsored by the McNair Scholars' Program and supported by
the National Science Foundation under Award No.~EPS-0903806 and
matching support from the State of Kansas through the Kansas Board of
Regents; further funding was received from a scholarship provided by
the College of Arts and Sciences at Kansas State University.

\appendix

\section{Coefficients of new SBP operators}
\applabel{coeffs} The coefficients for the new $6$th- and $7$th-order
SBP operators described in \secref{opt} are included online as text
files.  The files \texttt{P\_6\_6\_15.txt} and
\texttt{D\_6\_6\_15.txt} hold the coefficients for $P_h$ and $D_h$ of
the $6$th order method, respectively.  While \texttt{P\_7\_7\_19.txt}
and \texttt{D\_7\_7\_19.txt} hold the coefficients for $P_h$ and $D_h$
of the $7$th order method.  

The coefficients are scaled to the case $h=1$.  In the case of the
norm matrices $P_h$, the data are stored in rows of 2 columns.  Each
row holds a pair $(i,v)$ indicating that $\tilde{P}_{ii}=v$.  Only the
upper-left corner is given, since the lower-right corner can be
obtained from symmetry.  For the files containing $D_h$ coefficients,
each row contains a triple $(i,j,v)$ indicating that $d_{ij}=v$.
Again, only the upper-left corner is given.

\bigskip

\bibliographystyle{acm}
\biboptions{numbers,sort&compress}
\bibliography{AK_2015}

\end{document}